\documentclass[reqno,a4paper, 11pt]{amsart}

\usepackage[a4paper=true,pdfpagelabels]{hyperref}
\usepackage{graphicx}

\usepackage{latexsym}

\usepackage{amsmath}
\usepackage{amssymb}
\usepackage[T1]{fontenc}
\usepackage[utf8]{inputenc}
\usepackage{lmodern}
\usepackage{verbatim}

\newcommand{\hol}{{\mathcal Hol}}

\newcommand{\D}{\mathbb D}

\newtheorem{theorem}{Theorem}

\newtheorem{corollary}{Corollary}
\newtheorem{proposition}{Proposition}
\theoremstyle{definition}

\theoremstyle{remark}
\newtheorem{remark}{Remark}
\numberwithin{equation}{section}

\DeclareMathOperator{\ARG }{arg}

\newcommand{\HL}{\mathcal {HL}} 

\begin{document}

\title[The essential norm, block sizes, and some Generalized Hilbert operators...]
{The essential norm, block sizes, and some Generalized Hilbert operators on $\ell^p$}

\author[ ]{David Norrbo}

\address{Faculty of Science and Engineering, Åbo Akademi University, 20500 Åbo, Finland}
\email{david.norrbo@ntnu.no}





\begin{abstract}
We examine the generalized Hilbert (matrix) operators
$$
H_{g,\gamma} : (a_n) \mapsto \sum_{n=1}^{\infty}   \bigg(\frac{k}{n}\bigg)^{\gamma} \frac{g_k a_n}{n+k}
$$
on the $\ell^p$ spaces, $1<p<\infty$, where $g=(g_n)$ is a sequence and $-1/p<\gamma <1-1/p$. Given a partition of the natural numbers $\bigcup_j I_j = \mathbb N$, we encode $g$ to our investigation via its $\ell^p$-means $G_j$ over the sets $I_j$.  We prove that if $I_j$ increases exponentially, the data $(G_j)$ characterizes boundedness, but is insufficient to determine the exact value of the essential norm. Nonetheless, if the growth rate of $(I_j)$ is subexponential, the corresponding data $(G_j)$, given that the tail converge, is sufficient to determine the exact value, in which case we also calculate it. Moreover, it is shown that the boundedness of $(G_j)$ is sufficient, but not necessary (and necessary, but not sufficient) to ensure $H_{g,\gamma}$ is bounded if the underlying partition $(I_j)$ increases subexponentially (and superexponentially, respectively). Using the dual operator, we also prove that the essential norm of $H_{g,\gamma}$ is comparable with $\limsup_j G(j)$ when $(I_j)$ grows exponentially.

\end{abstract}

\maketitle

\section{introduction}
The Hilbert  operator 
\[
H_g(f)(z) = \int_0^1 f(t) g'(zt) \, dt,\quad f,g\in\hol(\D), \, z\in \D
\] 
was introduced in \cite{Galanopoulos-2014}, where the boundedness was was characterized on the Hardy spaces $H^p, \, 1<p\leq 2$ (Theorem 1), the standard weighted Bergman spaces, $A_\alpha^p,\, p-2> \alpha >-1$ (Theorem 3), and the Dirichlet type spaces $\mathcal D_\alpha^p,\, p-1\geq  \alpha >p-2$ (Theorem 4). In all four cases the boundedness is characterized by $g\in \Lambda(p,1/p)$, the mean Lipschitz space. The investigation on this operator has since been continued, see for example, \cite{Pelaez-2013,Pelaez-2018,Blasco-2022,Guo-2026,Norrbo-2026}, and \cite{Yicen-2026}. One of the common norms used on the mean Lipschitz spaces $\Lambda(p,1/p)$ is given by
\[
\|g\|_{\Lambda(p,\frac{1}{p})} \asymp |g(0)| + |g'(0)| + \sup_{n\in\mathbb N} 2^{n(\frac{1}{p}-1)} \| \Delta_n (g')\|_{H^p}
\]
where $(\Delta_n (g'))(z) = \sum_{k=2^{n-1}}^{2^{n+1}-1} (g')_k z^k $, is the canonical projections of $g'$ onto the finite-dimensional subspaces spanned by $\{z^k : k=2^{n-1},\ldots, 2^{n+1}-1 \},\, n\in\mathbb N$.  Although the derivative of $g$ has been used as the kernel in most of  the previous literature, the convention of writing $g$ instead of $g'$ will be adopted in this article.

We have turned our attention to the reflexive $\ell^p$ spaces ($1<p<\infty$), where the operator is given by
\[
H_g \colon  a \mapsto \sum _{n=0}^\infty \frac{g_k a_n}{n+k+1},
\]
where $a=(a_n),g=(g_k)$ are sequences. Formally, the relation follows from
\[
\sum _{n=0}^\infty \frac{g_k a_n}{n+k+1} = \int_0^1  \sum_{n=0}^\infty a_n t^n  g_k t^k \, dt
\]
in the view of the Taylor coefficients to an analytic function, that is,
\[
H_g\colon \sum_{n=0}^\infty a_n z^n  \mapsto  \sum_{k=0}^\infty \bigg(       \int_0^1  \sum_{n=0}^\infty a_n t^n  g_k t^k \, dt   \bigg)  z^k =      \int_0^1  \bigg(\sum_{n=0}^\infty a_n t^n \bigg)  \bigg( \sum_{k=0}^\infty   g_k (zt)^k   \bigg)\, dt. 
\]

We examine boundedness and the essential norm of $H_g \colon \ell^p \to \ell^p$. In comparison with earlier work, the focus lies on block sizes. We show that the boundedness of
\[
n\mapsto G(n;(2^j);g):=\frac{1}{2^{n+1}-2^n}  \sum_{k=2^n} ^{2^{n+1}-1} |g_k|^p,
\]
is equivalent to $H_g \colon \ell^p \to \ell^p$ being bounded. Moreover, the exponential base $s=2$ can be exchanged with any real number $1<s<\infty$ without altering the statement. However, we also prove that the data $G(n;(2^j);g), \, n\in \mathbb N$, is not sufficient to determine the exact value of the essential norm. Assume that we instead use blocks whose size increase subexponentially, for example, using the partition sequence $(j^2)$, we have
\[
G(n;(j^2);g):=\frac{1}{  (n+1)^2 - n^2    }  \sum_{k=n^2} ^{(n+1)^2-1} |g_k|^p.
\]
Under the assumption that $\lim_{n\to \infty} G(n;(j^2);g)$ exists, we calculate the exact value of the essential norm, where the data  $\lim_{n\to \infty} G(n;(j^2);g)$ is sufficient considering the contribution of $g$; there is also a $p$ dependency. Moreover, (as is necessary from the statement of the result) the limit is, whenever it exists, independent of the particular choice of partition sequence with subexponential growth. Moreover, for a subexponential sequence $d$, the boundedness of $n\mapsto  G(n;d;g)$, is sufficient, but not necessary for $H_g \colon \ell^p \to \ell^p$ to be bounded. 

Most of our results are in fact written for a slightly different, but in our context, more general operator, which allow us to transfer the result to the Hardy-Littlewood spaces. On these spaces, the exact value of the norm (and implicitly of the essential norm) of the classical Hilbert operator (where the sequence $g$ is constant), was obtained in \cite{Daskalogiannis-2025}. 

\section{Preliminaries}\label{sec:Prelim}
Our main results are Theorems \ref{thm:bdd}, \ref{thm:exactEssNorm} and \ref{thm:essentialNormEstimate} with their illustrative corollaries.  The first concerns boundedness, and the two latter, the essential norm. We also highlight Remarks \ref{rem:Gen} and \ref{rem:GenAdjoint} for the generality of the results. In Section \ref{sec:Prelim} we introduce the main definitions and state the main results, with the exception of Theorem \ref{thm:essentialNormEstimate}. Section \ref{sec:Proofs} contains the proofs of the results stated in the previous section. Finally, in Section \ref{sec:DualOp}, all the results involving the dual operator of $H_{g,\gamma}$ can be found, including Theorem \ref{thm:essentialNormEstimate}, which also concerns $H_{g,\gamma}$. 

Throughout this paper, we assume that $1<p<\infty$, $1/p +1/p' = 1$ and $-\frac{1}{p}<\gamma<1-\frac{1}{p}$ are fixed numbers, which will not be stated in the results. Moreover, $d:=(d_n)$ will be a strictly increasing sequence with $d_0=1$. 
Given a sequence $g:=(g_k)$, (can be compared with Taylor coefficients), we define
\[
G(n;d;g):= \frac{1}{d_n-d_{n-1}}  \sum_{k=d_{n-1}} ^{d_n-1} |g_k|^p.
\]
When there's no risk of ambiguity, we write $G(n)$. The notations $\hat{g}$ and $\hat{G}$ will be used similarly to $g$ and $G$, respectively, and a sequence $r$ will be used similarly to $d$. The operator of interest in this note is the following:
\[
H_{g,\gamma} \colon \ell^p \to \ell^p ,\quad H_{g,\gamma} (a) := \sum_{n=1}^{\infty}   \bigg(\frac{k}{n}\bigg)^{\gamma} \frac{g_k a_n}{n+k}.
\]
We will use $a:=(a_n)$ and $b:=(b_n)$ for sequences of complex numbers and $\alpha:=(\alpha_m)$ for sequences of sequences of complex numbers. For a sequence $g$, we define $|g|:= (|g_k|)$.

For convenience, for $a\in \ell^p$ and $K\in\mathbb N$,  we define 

\begin{equation}\label{eq:basicDuality} 
C_K(g;a) := \Bigg(\sum_{k=d_{K-1}}^{\infty}  \Bigg|  \sum_{n=1}^{\infty}   \bigg(\frac{k}{n}\bigg)^{\gamma}  \frac{  a_n g_k  }{n+k} \Bigg|^p \Bigg)^{\frac{1}{p}} = \sup_{\|b\|_{\ell^{p'}} = 1} \sum_{k=d_{K-1}}^{\infty}   \sum_{n=1}^{\infty}   \bigg(\frac{k}{n}\bigg)^{\gamma} \frac{  a_n b_ k g_k  }{n+k}.
\end{equation}
Notice that the $\gamma$- and $p$-dependencies are omitted in the notations $G(n,d,g)$ and $C_K(g;a)$. We have also left out the $d$ dependency from $C_K(g;a)$, since it's not so essential, but it makes the calculations fit well together.

Recall that the essential norm of a bounded operator $S\colon \ell^p\to \ell^p$ is given by $\|S\|_{e,\ell^p\to\ell^p} :=\inf_T\|S-T\|_{\ell^p\to\ell^p}  $, where the infimum is taken over all compact operators $T\colon\ell^p\to\ell^p$. The essential norm for an operator on $\ell^p, \, 1<p<\infty$ coincides with the weak null maximum, defined as
\[
\textrm{wem}(S) := \textrm{wem}(S;\ell^p):=   \sup_{\alpha\in W(B_{\ell^p})}  \limsup_m \| S(\alpha_m) \|_{\ell^p},
\]
where $W(B_{\ell^p})$ is the set of weakly null sequences of elements in the unit ball of $\ell^p$, $B_{\ell^p}$. This can easily be seen by using the complement projection to the finite rank canonical projection, that is, 
\begin{equation}\label{eq:projQ}
Q_k\colon \ell^p\to \ell^p,  \sum_{n=1}^\infty a_n e_n \mapsto \sum_{n=k}^\infty a_n e_n, 
\end{equation}
which satisfies $\| Q_k \|_{\ell^p\to \ell^p} = 1$, where $k\in\mathbb N $. Indeed, for every $k\in\mathbb N$, let $0\neq \alpha_k\in B_{Q_k(\ell^p)} =Q_k(B_{\ell^p})\subset B_{\ell^p}$ be such that
\[
\sup_{a\in B_{\ell^p}} \|  S (Q_k (a)) \| _{\ell^p} -  \|  S (Q_k (\alpha_k)) \| _{\ell^p} < \frac{1}{k}. 
\]
Since $Q_k \alpha_k = \alpha_k$, we have obtained a weak null sequence $(\alpha_k)\subset B_{\ell^p }$ such that
\[
\| S \|_{e,\ell^p\to\ell^p}  = \liminf_k\| SQ_k \|_{e,\ell^p\to\ell^p}  \leq  \liminf_k \sup_{a\in B_{\ell^p}} \|  S (Q_k (a)) \| _{\ell^p}  \leq    \limsup_k \|  S (\alpha_k) \| _{\ell^p} \leq  \textrm{wem}(S). 
\]
 Since  $ \textrm{wem}(S)$, by definition, is dominated by $\| S \|_{e,\ell^p\to\ell^p} $, we are done. The fact that the essential norm equals the weak null maximum can more generally be proven for $(M_p)$-spaces when $1<p<\infty$, see for example \cite{Werner-1992}.

The Hardy-Littlewood space $\HL^p$ $(1<p<\infty)$ can be represented as the complex sequence space generated by the norm
\[
\|a\|_{\HL^p}:= \bigg(  \sum_{n=1}^\infty  a_n^p  n^{p-2}   \bigg)^{\frac{1}{p}}. 
\]

For two nonnegative functions $f$ and $g$ (possibly dependent of some parameters), we write $f\gtrsim g$ with respect to $\eta$ if there is a constant $C$, independent of $\eta$ such that $f(x)\geq C g(x)$ for all $x$ in some implicit set. If $f\gtrsim g$ and $g\gtrsim f$, we write $f\asymp g$.

We are now ready to state the main results. We remark that the factor $(n+k)^{-1}$ could be replaced by $(n+k-1)^{-1}$, which corresponds to the classical generalized Hilbert matrix (see Remark \ref{rem:Gen}).

\begin{theorem}[Boundedness]\label{thm:bdd}
It holds that for any $K\in\mathbb N$ and $a_n\geq 0$ for all $n\in \mathbb N$
\[
C_K(g;a) \leq \sup_{j\geq K} G(j)^{\frac{1}{p}}   \max\bigg\{   \bigg(\frac{d_j}{d_{j-1}}  \bigg)^{1-\gamma}  ,   \bigg(\frac{d_j}{d_{j-1}}  \bigg)^{\gamma} \bigg\} C_K(\hat{g};a),
\]
and 
\[
 \sup_{\|a\|_{\ell^{p}} = 1} C_K(g;a)  \geq  \sup_{j\geq K} G(j)^{\frac{1}{p}} \bigg( \frac{d_j - d_{j-1}}{ d_j  }  \bigg)^{\frac{1}{p}} \bigg(  \int_3^\infty  \frac{x^{-|\gamma|p'}}{  (1+x)^{p'} }  \, dx \bigg)^{\frac{1}{p'}},
\]
 where $\hat{g}$ is any sequence with $\hat{G}\equiv 1$. 

Moreover, 
\begin{enumerate}
\item $\sup_n G(n;d;g)<\infty$ implies $H_{g,\gamma}$ is bounded on $\ell^p$ if and only if $\sup_j d_j/d_{j-1} <\infty$;\label{eq:GfiniteImply}
\item$H_{g,\gamma}$ is bounded on $\ell^p$ implies $\sup_n G(n;d;g)<\infty$ if and only if $\inf_j d_j/d_{j-1} >1$. 
\label{eq:HBddImply}
\end{enumerate}

\end{theorem}

\begin{corollary}\label{cor:CharactBdd}
Let $\epsilon>0$. $H_{g,\gamma}$ is bounded on $\ell^p$ if and only if $\sup_n G(n;((1+\epsilon)^j);g)<\infty$. Furthermore,
\[
\|H_{g,\gamma}\|_{\ell^p\to\ell^p} \asymp \sup_n G(n)
\]
with respect to $g$.
\end{corollary}

\begin{theorem}[Essential norm]\label{thm:exactEssNorm}
If  $\lim_k d_k/d_{k-1} = 1$ and $\lim_n G(n;(d_j),g) \in [0,\infty[$, then
\[
\|H_{g,\gamma}\|_{e,\ell^p\to\ell^p} =     \frac{\pi}{\sin\Big(  \pi\Big(\frac{1}{p} + \gamma \Big) \Big) }  \big(\lim_n G(n)\big)^{\frac{1}{p}}.
\]
If $\lim_k d_k/d_{k-1} > 1$, there exists no function $F\colon [0,\infty[\to [0,\infty[$, independent of the individual elements in the sequence $g$, such that
\[
\|H_{g,\gamma}\|_{e,\ell^p\to\ell^p} = F(G).
\]
\end{theorem}

The choice $\gamma =  1 - 2/p$, gives us
\begin{corollary}\label{cor:HLnorm}
The operator $H_{g,0}$ is bounded on $\HL^p$ if and only if $\sup_n G(n;(2^j);g)<\infty$, in which case the norm satisfies
\[
\|H_{g,0}\|_{\HL^p \to \HL^p} \asymp \sup_n G(n)^{\frac{1}{p}},
\]
with respect to $g$.

Moreover, the statements of Theorem \ref{thm:exactEssNorm} also hold with $\gamma=0$ and $\ell^p$ replaced by $\HL^p$.
\end{corollary}

We remark that if $\lim_k d_k/d_{k-1} = 1$ and $\lim_n G(n;d;g)=:L \in [0,\infty[$, then the value of the limit is independent of the choice of sequence $d$. This can be seen by first proving that $\lim_n G(n;d;g)=:L$ implies 
\[
\lim_n \frac{1}{d_n} \sum_{j=1}^{ d_n-1 } |g_j|^p = L,
\]
which holds without the condition $\lim_k d_k/d_{k-1} = 1$. Then impose this condition to prove that the supscript in the sum can be changed to any $k_n\in[d_{n-1},d_n]$, without altering the value of the limit. It follows that 
\[
\lim_n \frac{1}{n} \sum_{j=1}^{ n } |g_j|^p = L,
\]
which is independent of the particular choice $d$. If another sequence $r$, satisfying the same assumptions, would give a different value, we would obtain a contradiction, and hence, we are done.

We also mention that it follows from Theorem \ref{thm:bdd} that 
\[
\{g : \sup_n G(n,d,g)<\infty \text{ for some } d \text{ with } \lim_j d_j/d_{j-1} = 1 \}
\]
is a strict subset of $\{g : \sup_n G(n,(2^j),g)<\infty \}$.

Before we continue to the proofs, notice that by \eqref{eq:basicDuality}, it is sufficient to assume $g_n\geq 0$, after which we can also assume $a_n,b_n\geq 0$ for every $n$, due to the argument invariance when calculating the norm, essential norm, or weak null maximum (triangle inequality with modulus for the upper bound, and reducing the set we take supremum over for the lower bound). In particular, $C_K(g;a) = C_K(|g|;a)$. Given that $S$ is bounded on $\ell^p$ and $\alpha\in W(B_{\ell^p})$, $m\mapsto S(\alpha_m)$ will also be weakly null, hence, altering the starting index of the sum in $H_{g,\gamma}$ and the sum in the norm does not change the value of the essential norm. More generally, we have the following:

\begin{remark}\label{rem:Gen}
The operators 
\[
\ell^p \to \ell^p,\quad   (a_n) \mapsto \bigg[ k\mapsto   \sum_{n=1}^{\infty}   \bigg(\frac{k+l_2}{n+l_3}\bigg)^{\gamma} \frac{  a_{n+l_4} g_{k+k_1}  }{n+k+l_1}  \bigg],
\]
where $-2< l_1<\infty$, $-1<l_2,l_3<\infty$ and $l_4,k_1\in \mathbb N\cup \{0\}$ have the same essential norm, weak null maximum, and criteria for boundedness as $H_{g,\gamma}$ on $\ell^p$.
\end{remark}

Since the lower bound for the generalized Hilbert matrix operator with $g\equiv 1$ is obtained using a weakly null sequence, we have:

\begin{proposition}\label{prop:The CasegEquiv1}
For every $K\in\mathbb N$, we have
\[
\textrm{wem}(H_{1,\gamma}) = \sup_{\|a\|_{\ell^p}=1}C_K(1;a) = \| H_{1,\gamma} \|_{e,\ell^p\to\ell^p} = \| H_{1,\gamma} \|_{\ell^p\to\ell^p} =  \frac{\pi}{\sin\Big(  \pi\Big(\frac{1}{p} + \gamma \Big) \Big) }.
\]

\end{proposition}
For a proof, we refer to \cite[Theorem 3.1.]{Yang-2005} in which we use $\lambda=\alpha=1$ and $\phi_p= 1-1/p-\gamma $, and a substitution of $a_n$ and $b_n$ to obtain $\ell^p$-norms on the right-hand side.

\section{Proofs}\label{sec:Proofs}

We begin to note that Corollary \ref{cor:CharactBdd} follows immediately from Theorem \ref{thm:bdd}. Concerning Corollary \ref{cor:HLnorm}, we note that any weighted $\ell^p$, with $\|a\|^p =\sum_n w_n a_n^p$, where $w_n> 0$ for all $\mathbb N$ is isometrically isomorphic to $\ell^p$. Therefore, the essential norm and the weak null maximum coincides. The statements now follows by applying the isometric isomorphism $\HL^p \to \ell^p$, $a_n \mapsto a_n n^{2/p-1}$ and note that $\sin\big(\pi\big(  1/p  + 1 - 2/p \big)\big) = \sin\big(\pi\big(  1/p  \big)\big) $.  Henceforth, we assume $a_n,b_n,g_n \geq 0$ for all $n\in\mathbb N$.

Define
\[
B_j:= \bigg( \sum_{k=d_{j-1}}^{d_j-1} b_j^{p'} \bigg)^{\frac{1}{p'}} \quad \text{ and } \quad  A_k:= A_k((a_n);\gamma) := \sum_{n=1}^{\infty} \bigg( \frac{k}{n} \bigg)^{\gamma}\frac{a_n}{n+k}. 
\]
Let
\[
k_{j,\min} := \ARG\min\{ A_k  :  k\in [d_{j-1},d_j) \} \quad \text{ and } \quad k_{j,\max} := \ARG\max\{ A_k  :  k\in [d_{j-1},d_j) \}
\]
Notice that if $\gamma\leq 0$, then $A_k$ is decreasing, but for $\gamma>0$ it is not clear, which one of $k_{j,\max} $ and $k_{j,\min} $ is bigger. Since we can assume $a_n\geq 0$ for all $n\in\mathbb N$, we have
\begin{equation}\label{eq:estimateAk} 
\begin{split}
1&\leq  \frac{  A_{k_{j,\max}}   }{       A_{k_{j,\min}}     } =   \frac{     \sum_{n=N}^{\infty}   \big( \frac{ k_{j,\max} }{n} \big)^{\gamma}   \frac{a_n}{ n+k_{j,\min} }        \frac{   n+k_{j,\max}  + k_{j,\min} - k_{j,\max}  }{     n+k_{j,\max}    }     }{          \sum_{n=N}^{\infty} \big( \frac{k_{j,\min}}{n} \big)^{\gamma}   \frac{a_n}{n+k_{j,\min}}      }\\
&\leq  \bigg( \frac{ k_{j,\max} }{ k_{j,\min} } \bigg)^{\gamma}  \sup_n \bigg( 1  +  \frac{ k_{j,\min} - k_{j,\max}   }{ n + k_{j,\max}  } \bigg) = \max\bigg\{  \bigg(  \frac{k_{j,\min}  }{ k_{j,\max} }  \bigg)^{1-\gamma}  ,   \bigg(  \frac{ k_{j,\max} }{k_{j,\min}  }  \bigg)^{\gamma}   \bigg\}   \\
&\leq   \max\bigg\{   \bigg(\frac{d_j}{d_{j-1}}  \bigg)^{1-\gamma}  ,   \bigg(\frac{d_j}{d_{j-1}}  \bigg)^{\gamma} \bigg\}.
\end{split}
\end{equation}

Another important estimate is the following: 
\begin{equation}\label{eq:leqq} 
 \sup_{ \|B\|_{\ell^{p'}} \leq 1   } \sum_{j=K}^{\infty}  B_j\big( (d_j- d_{j-1})G(j)  \big)^{\frac{1}{p}}   A_{k_{j,\min}}   \leq C_K(g;a) \leq   \sup_{ \|B\|_{\ell^{p'}} \leq 1  }  \sum_{j=K}^{\infty}  B_j\big( (d_j- d_{j-1})G(j)  \big)^{\frac{1}{p}}  A_{k_{j,\max}}.
\end{equation}

To see that \eqref{eq:leqq} holds, for $B_j\geq 0, \, j\in \mathbb N$, the supremum over $b$ in the expression for $C_K$ can be manipulated as follows:
\begin{align*}
\sup_{\|(b_n)_{n=1}^{\infty}\|_{\ell^{p'}} \leq 1} \sum_{k=d_{K-1}}^{\infty} b_k g_k A_k &= \sup_{\|(B_n)_{n=K}^{\infty}\|_{\ell^{p'}} \leq 1}  \sup_{ \substack{   \|(b_{d_{n-1}},\ldots, b_{d_n-1})\|_{\ell^{p'}}   \leq  B_n   \\  n=K,\ldots, \infty }  }  \sum_{j=K}^{\infty} \sum_{k=d_{j-1}}^{d_j-1} b_k g_k A_k \\
&=  \sup_{\|(B_n)_{n=K}^{\infty}\|_{\ell^{p'}} \leq 1} \sum_{j=K}^{\infty}  \sup_{ \|(b_{d_{j-1}},\ldots, b_{d_j-1})\|_{\ell^{p'}} \leq  B_j  } \sum_{k=d_{j-1}}^{d_j-1}  b_k g_k A_k ,
\end{align*}
since the only dependence on $b$ is via $b_k$, hence, only one term in the leftmost/outer sum (in the double sum expressions) is dependent on the parameters $(b_{d_{n-1}},\ldots, b_{d_n-1})$ in the supremum for a fixed $n$. Moreover,
\[
 \sup_{ \|(b_{d_{j-1}},\ldots, b_{d_j-1})\|_{\ell^{p'}} \leq  B_j  }  \sum_{k=d_{j-1}}^{d_j-1}  b_k g_k A_k =  B_j \bigg(  \sum_{k=d_{j-1}}^{d_j-1}  g_k^p A_k^p   \bigg)^{\frac{1}{p}}.
\]
The lower bound now follows from
\begin{equation*}\label{eq:leqq1} 
\bigg(  \sum_{k=d_{j-1}}^{d_j-1}  g_k^p A_k^p   \bigg)^{\frac{1}{p}}  \geq    \bigg(  \sum_{k=d_{j-1}}^{d_j-1}  g_k^p A_{k_{j,\min}}^p   \bigg)^{\frac{1}{p}}  =  \big( (d_j- d_{j-1})G(j)  \big)^{\frac{1}{p}}  A_{k_{j,\min}}.  
\end{equation*}
The upper estimate is shown in a similar way.

Combining \eqref{eq:estimateAk}  and \eqref{eq:leqq} yields 
\[
C_K(g;a) \leq \sup_{j\geq K} G(j)^{\frac{1}{p}}   \max\bigg\{   \bigg(\frac{d_j}{d_{j-1}}  \bigg)^{1-\gamma}  ,   \bigg(\frac{d_j}{d_{j-1}}  \bigg)^{\gamma} \bigg\} C_K(\hat{g};a),
\]
 where $\hat{g}$ is any sequence with $\hat{G}\equiv 1$.

\begin{proof}[Proof of Theorem \ref{thm:bdd}]

We have already established the upper estimate, which yields the ``if'' part of \eqref{eq:GfiniteImply}. For the lower bound, choose $J\in \mathbb N$ with $J\geq K$. We put $B_J=1$ and $B_j = 0$ for $j\in \mathbb N\setminus \{J\}$, that is, the Jth unit vector. We obtain
\begin{equation}\label{eq:LowerBoundForNorm}
\begin{split}
 \sup_{  \|a\|_{\ell^p} = 1 }  \sum_{j=K}^{\infty}  B_j\big( (d_j- d_{j-1}) & G(j)  \big)^{\frac{1}{p}}  A_{k_{j,\min}}  =  \sup_{  \|a\|_{\ell^p} = 1 } \big( (d_J- d_{J-1})G(J)  \big)^{\frac{1}{p}}  A_{k_{J,\min}}   \\
&=  \big( (d_J- d_{J-1})G(J)  \big)^{\frac{1}{p}}   \Bigg(    \sum_{n=1}^{\infty}   \bigg(  \bigg( \frac{ k_{J,\min}  }{n} \bigg)^{\gamma}\frac{ 1 }{n+k_{J,\min}}     \bigg)^{p'}        \Bigg)^{\frac{1}{p'}}.
\end{split}
\end{equation}

It remains to estimate 
\[
  \sum_{n=1}^{\infty}  \frac{   n^{-\gamma {p'}}   }{(n+k_{J,\min})^{p'}}.   
\]
If $0\leq \gamma <1-1/p$ the terms are decreasing and a lower bound is given by
\[
\int_1^{\infty}   \frac{   x^{-\gamma {p'}}  \, dx  }{(x+k_{J,\min})^{p'}} =  k_{J,\min}^{-({p'}\gamma+ {p'}-1)}  \int_{\frac{1}{ k_{J,\min} }  }^{\infty}   \frac{   x^{-\gamma {p'}}  \, dx  }{(x+1)^{p'}} . 
\]
If $-1/p<\gamma < 0$, $x\mapsto \frac{   x^{-\gamma {p'}}  \, dx  }{(x+k_{J,\min})^{p'}}$ is first increasing and then decreasing on $(0,\infty)$. The maximum is attained at $x=\frac{\gamma k_{J,\min}}{1-\gamma}$; therefore, the sum will dominate the integral if we on the increasing part compare the integral with rectangles whose height is given by the right most value and on the decreasing part compare with the leftmost. To be able to do this pairing we must exclude from the path of integration a suitable large neighborhood of the maximum of the integrand, which here has here been chosen to be $[\frac{\gamma k_{J,\min}}{1-\gamma} -1, \frac{\gamma k_{J,\min}}{1-\gamma} +1]$. 

It follows that
\[
  \sum_{n=1}^{\infty}  \frac{   n^{-\gamma {p'}}   }{(n+ k_{J,\min} )^{p'}} \geq  k_{J,\min} ^{-({p'}\gamma+ {p'}-1)}  M_L(p,\gamma,d,J),
\]
where
\[
M_L(p,\gamma,d,J) := \begin{cases}
\int_{  \frac{1}{ k_{J,\min} }   }^{\infty} \frac{   x^{-\gamma {p'}}\, dx  }{  (x+1)^{p'}} &  \gamma\geq  0\\
\Bigg(   \int_{   \frac{1}{ k_{J,\min} }     }^{\frac{\gamma }{ 1-\gamma } - \frac{1}{ k_{J,\min} }    }  +    \int_{  {\frac{\gamma }{ 1-\gamma } +  \frac{1}{ k_{J,\min} }   }     }^{\infty}\Bigg)  \frac{   x^{-\gamma {p'}}\, dx  }{ (x+1)^{p'}} & \gamma< 0 ;
\end{cases} 
\]
the value of the integrals should be interpreted as zero if the integral path is negative. Combining this with \eqref{eq:LowerBoundForNorm} and \eqref{eq:leqq}, we have now proved that
\[
\sup_{  \|a\|_{\ell^p} = 1 } C_K(g;a) \geq   G(J)^{\frac{1}{p}} \bigg( \frac{d_J - d_{J-1}}{ k_{J,\min}  }  \bigg)^{\frac{1}{p}}  M_L(p,\gamma,d,J)^{1-\frac{1}{p}}
\]
for every integer $J\geq K$. Since $M_L(p,\gamma,d,J)^{1-\frac{1}{p}}\asymp 1$ with respect to $J\in\mathbb N$, we see that $\sup_{J\geq 1} G(J)<\infty$ is necessary for $H_{g,\gamma}$ to be bounded on $\ell^p$ if 
\[
\liminf_{J\to\infty}\frac{d_J - d_{J-1}}{ k_{J,\min}  } >0.
\]
Since $d$ is strictly increasing, the $\liminf_J$ can be exchanged with $\inf_J$. The ``if'' part of the statement \eqref{eq:HBddImply} now follows from the estimate
\[
1 -  \frac {d_{J-1}}{d_J}   \leq  \frac{d_J - d_{J-1}}{ k_{J,\min}  } \leq \frac {d_J}{d_{J-1}}  - 1. 
\]

To prove the ``only if'' part of \eqref{eq:GfiniteImply}, consider the function $g$ defined by $g_{k} = (d_{j+1}- d_j)^{1/p}$, when $k=d_j$ and zero otherwise. It is clear that $G \equiv 1$. Moreover,
\[
C_1(g;a)  =  \sup_{ \|B\|_{\ell^{p'}} \leq 1   } \sum_{j=1}^{\infty} B_j \big( (d_j- d_{j-1})G(j)  \big)^{\frac{1}{p}}  A_{d_{j-1}}.
\] 
Therefore, from the same calculations used for the lower bound, we obtain
\[
\sup_{  \|a\|_{\ell^p} = 1 } C_1(g;a) \gtrsim   G(J)^{\frac{1}{p}} \bigg( \frac{d_J - d_{J-1}}{ d_{J-1}  }  \bigg)^{\frac{1}{p}}  =\bigg( \frac{d_J }{ d_{J-1}  } -1  \bigg)^{\frac{1}{p}},
\]
with respect to $J$, and the statement follows.

To prove the ``only if'' part of \eqref{eq:HBddImply}, consider the function $g$ defined as follows. Put $J_0=1$ and $g_{J_0}=1$. Then, for $k=1$, choose an index $J_k$ such that $d_{J_k} > 2d_{J_{k-1}}$ and  $d_{J_k}/d_{J_k-1} - 1 \leq 1/k$ (pay attention to which subscript has a ``$-1$''). Define $g_{d_{J_k}} = d_{J_k}^{\frac{1}{p}}$. Repeat the process for all $k\in\mathbb N$. Define $g$ to be zero elsewhere.

Define $r_0= d_{J_0}$ and $r_1=2d_{J_0}$. Then define $r_2=2d_{J_0}+1,r_3=2d_{J_0}+2,\ldots, r_{m_1} = d_{J_1}$. Furthermore, $r_{m_1+1} = 2d_{J_1}$ and then we continue with blocks of size one until we hit the value $d_{J_2}$. Now $\sup _k r_k/r_{k-1} \leq 2$, and 
\[
\frac{1}{r_k-r_{k-1}} \sum_{j=r_{k-1}}^{r_k-1} g_j^p
\]  
contains at most one term that is nonzero.  Therefore,
\[
\sup_m G(m;r;g) = \sup_m\frac{1}{r_m-r_{m-1}} \sum_{j=r_{m-1}}^{r_m-1} g_j^p \leq \sup_k  \frac{  \Big( d_{J_k}^{\frac{1}{p}}  \Big)^p     }{ 2d_{J_k}   - d_{J_k}     } = 1,
\]
and by \eqref{eq:GfiniteImply}, $H_{g,\gamma}$ is bounded on $\ell^p$. However,
\[
\frac{1}{d_{J_k} - d_{J_k-1}  } \sum_{j=d_{J_k-1} }^{d_{J_k} -1} g_j^p = \frac{  d_{J_k-1}  }{d_{J_k} - d_{J_k-1}  } \geq  k,
\]
which proves $n\mapsto G(n;d;g)$ is not bounded.

\end{proof}

\begin{proof}[Proof of Theorem 2]

First, assume $d_j/d_{j-1}\sim 1$ and that $\lim_k G(k)$ exists. We will prove that 
\[
\|H_{g,\gamma}\|_e= \big(  \lim_k G(k) \big)  \frac{\pi}{\sin \Big( \pi\Big(\frac{1}{p}+\gamma \Big)\Big) }.
\]

For the lower bound, we will make use of the projection $R_K:=Q_{d_{K-1}}$, where $Q_{d_{K-1}}$ is the complement canonical projection, defined in \eqref{eq:projQ}. Recall that in the paragraph right before Remark \ref{rem:Gen}, it was mentioned that the weakly null maximum ignores the starting index of the norm, that is,  $\textrm{wem}(R_K S;\ell^p) = \textrm{wem}(S;\ell^p)$ for every $K\in\mathbb N$. For $(a_n)\in\ell^p$, we have
\begin{equation}
\begin{split}
\Big( \inf_{j\geq K} G(j) \Big) \|R_K H_{1,\gamma} (a)\|_{\ell^p}&\leq \Big( \inf_{j\geq K} G(j) \Big) \sup_{ \|B\|_{\ell^{p'}} = 1 } \sum_{j=K}^{\infty}  B_j(d_j- d_{j-1})^{\frac{1}{p}}A_{k_{j,\max}} \\
&\leq   \sup_{ \|B\|_{\ell^{p'}} = 1 } \sum_{j=K}^{\infty}  B_j\big( (d_j- d_{j-1})G(j)  \big)^{\frac{1}{p}}  \frac{     A_{k_{j,\max}}    }{   A_{k_{j,\min}}     }A_{k_{j,\min}} \\
&\leq \sup_{j\geq K} \max\bigg\{   \bigg(\frac{d_j}{d_{j-1}}  \bigg)^{1-\gamma}  ,   \bigg(\frac{d_j}{d_{j-1}}  \bigg)^{\gamma} \bigg\} \|H_{g,\gamma}(a)\|_{\ell^p}.
\end{split}
\end{equation}
Considering weakly null sequences, we obtain for any $K\in\mathbb N$
\[
\Big( \inf_{j\geq K} G(j) \Big) \textrm{wem}(H_{1,\gamma}) \leq  \sup_{j\geq K}   \max\bigg\{   \bigg(\frac{d_j}{d_{j-1}}  \bigg)^{1-\gamma}  ,   \bigg(\frac{d_j}{d_{j-1}}  \bigg)^{\gamma} \bigg\}  \|H_{g,\gamma}\|_{e,\ell^p\to \ell^p}
\] 
from which the statement follows from letting $K\to \infty$ together with Proposition \ref{prop:The CasegEquiv1}. The upper bound is similarly obtained from the estimate
\[
\|  (R_K H_{g,\gamma})(a) \|_{\ell^p} = C_K(g;a) \leq \sup_{j\geq K} G(j)^{\frac{1}{p}}   \max\bigg\{   \bigg(\frac{d_j}{d_{j-1}}  \bigg)^{1-\gamma}  ,   \bigg(\frac{d_j}{d_{j-1}}  \bigg)^{\gamma} \bigg\}  C_K(\hat{g};a),
\]
given in Theorem \ref{thm:bdd}, using $\hat{g}\equiv 1$. The above holds when $a_n\geq 0 $ for all $n\in\mathbb N$, but on the left-hand side the triangle inequality can be used to obtain the statement for general $a$, while the right-hand side gets larger if we don restrict the weak null sequences to only contain $a_n\geq 0$.

Assume now that $\lim_k d_k/d_{k-1} = :\mu >1$ and consider the sequence $g^{(+)}$ given by $g_k^{(+)}= d_j-d_{j-1}$ if $k = d_{j-1}, \, j\in\mathbb N$ and zero otherwise. Similarly, define $g_k^{(-)}= d_j-d_{j-1}$ if $k = d_j-1, \, j\in\mathbb N$ and zero otherwise. For the means, we have $G^{(\pm)}\equiv 1 $.

Furthermore, from the proof of \eqref{eq:leqq}, it is evident that for every $(a_n)\in\ell^p$, by leaving out a term to obtain the inequality below, we have 
\[
\| R_K H_{g^{(+)},\gamma} (a)\|_{\ell^p}  =  \sup_{\|B\|_{\ell^{p'}} \leq 1} \sum_{j=K}^{\infty} B_j   (d_j-d_{j-1})^{\frac{1}{p}}  A_{d_{j-1}} \geq    \sup_{\|B\|_{\ell^{p'}} \leq 1} \sum_{j=K}^{\infty} B_j   (d_{j+1}-d_j)^{\frac{1}{p}}  A_{d_j}, 
\]
and
\[
\| R_K H_{g^{(-)},\gamma} (a) \|_{\ell^p}   = \sup_{\|B\|_{\ell^{p'}} \leq 1} \sum_{j=K}^{\infty} B_j   (d_j-d_{j-1})^{\frac{1}{p}}  A_{d_j-1}.
\]
From a calculation similar to \eqref{eq:estimateAk}, it follows that
\[
\frac{   A_{d_j-1}   }{   A_{d_j}  }  \leq  \bigg(1-\frac{1}{d_j}\bigg)^{\gamma-1}.
\]
Moreover,
\[
\frac{  d_j-d_{j-1}  }{    d_{j+1}-  d_j }  = \bigg(1-\frac{d_{j-1}}{d_j}\bigg) \bigg/ \bigg(\frac{d_{j+1}}{d_j} -1\bigg)
\]
so
\[
\limsup_j \bigg( \frac{  d_j-d_{j-1}  }{    d_{j+1}-  d_j } \bigg)^{\frac{1}{p}} \frac{   A_{d_j-1}   }{   A_{d_j}  }   \leq  \bigg( \frac{1}{\mu}\bigg)^{\frac{1}{p}}  .
\]
We can, therefore, choose $K$ large enough to ensure that
\[
\| R_K H_{g^{(-)},\gamma} (a) \|_{\ell^p} \leq  \bigg(\frac{1+\mu^{-1}  }{2} \bigg)^{\frac{1}{p}}  \| R_K H_{g^{(+)},\gamma} (a)\|_{\ell^p}.
\]
Since $a\in\ell^p$ is arbitrary, we obtain the statement via the fact that the weak null maximum coincides with the essential norm on $\ell^p$-spaces. In other words, we have proved there exists two different operators $H_{g^{(+)},\gamma} $ and $H_{g^{(-)},\gamma} $ with the same means $G^{(+)}=G^{(-)} = 1$, but with different essential norms.
\end{proof}

\section{The dual operator}\label{sec:DualOp}
We finish by reminding the reader that due to $\ell^p, \, 1<p<\infty$ being reflexive, all the results presented above can be translated to the dual operator of $H_{g,\gamma}$,
\[
H^*_{g,\gamma} \colon \ell^{p'} \to \ell^{p'} ,\quad H^*_{g,\gamma} (b) := \sum_{k=1}^{\infty}   \bigg(\frac{k}{n}\bigg)^{\gamma} \frac{g_k b_k}{n+k},
\]
because $\|  H_{g,\gamma}  \|_{ \ell^{p} \to \ell^{p} } = \|  H^*_{g,\gamma}  \|_{ \ell^{p'} \to \ell^{p'} }$ and $\|  H_{g,\gamma}  \|_{e, \ell^{p} \to \ell^{p} } = \|  H^*_{g,\gamma}  \|_{e, \ell^{p'} \to \ell^{p'} }$. Moreover, as in Remark \ref{rem:Gen}, we have 
\begin{remark}\label{rem:GenAdjoint}
The operators
\[
\ell^{p'} \to \ell^{p'},\quad   (b_n) \mapsto   \bigg[ n \mapsto  \sum_{k=1}^{\infty}   \bigg(\frac{k+l_2}{n+l_3}\bigg)^{\gamma} \frac{  b_{k+l_4} g_{k+k_1}  }{n+k+l_1} \bigg]
\]  
where $-2< l_1<\infty$, $-1<l_2,l_3<\infty$ and $l_4,k_1\in \mathbb N\cup \{0\}$ have the same essential norm, weak null maximum, and criteria for boundedness as $H^*_{g,\gamma}$ on $\ell^{p'}$.
\end{remark}

We now present the last theorem, which in addition to previous results uses duality.

\begin{theorem}\label{thm:essentialNormEstimate}
It holds that
\[
\| H_{g,\gamma} \|_{e,\ell^p\to\ell^p} = \| H^*_{g,\gamma} \|_{e,\ell^{p'}\to\ell^{p'}} = \lim_{K\to\infty} \sup_{\|a\|_{\ell^p}=1} C_K(|g|;a).
\]
In particular, if $1<\liminf_j d_j/d_{j-1} \limsup_j d_j/d_{j-1}<\infty$, then
\[
\| H_{g,\gamma} \|_{e,\ell^p\to\ell^p} = \| H^*_{g,\gamma} \|_{e,\ell^{p'}\to\ell^{p'}}  \asymp \limsup_j G(j)^{\frac{1}{p}},
\]
with respect to $g$.
\end{theorem}
\begin{proof}
Recall that in the proof of Theorem \ref{thm:exactEssNorm}, we used the projection $R_K:=Q_{d_{K-1}}$, which will again be useful, see \eqref{eq:projQ} for the definition of $Q_{d_{K-1}}$. 

We begin by noting that $K\mapsto \sup_{\|a\|_{\ell^p}=1} C_K(|g|;a)$ is decreasing and bounded from below, hence the limit exists as $K\to\infty$. The inequalities in Theorem \ref{thm:bdd} yield $\lim_{K\to\infty} \sup_{\|a\|_{\ell^p}=1} C_K(|g|;a) \asymp \limsup_j G(j)^{\frac{1}{p}}$ with respect to $g$. 

We are left to prove that
\[
 \| H^*_{g,\gamma} \|_{e,\ell^{p'}\to\ell^{p'}} = \lim_{K\to\infty} \sup_{\|a\|_{\ell^p}=1} C_K(|g|;a).
\]
To this end, observe that
\[
 \| H^*_{g,\gamma} \|_{e,\ell^{p'}\to\ell^{p'}} = \lim_{K\to\infty} \| H^*_{g,\gamma} R_K \|_{e,\ell^{p'}\to\ell^{p'}} \leq \lim_{K\to\infty} \| H^*_{g,\gamma} R_K \|_{\ell^{p'}\to\ell^{p'}} = \lim_{K\to\infty} \sup_{\|a\|_{\ell^p}=1} C_K(|g|;a).
\]
To prove the reverse inequality, similarly to the calculations done beneath \eqref{eq:projQ}, for every $K\in\mathbb N$, pick $\alpha_K\in R_K B_{\ell^{p'}}$ such that
\[
\| H^*_{g,\gamma} R_K \|_{\ell^{p'}\to\ell^{p'}} -  \| (H^*_{g,\gamma} R_K) (\alpha_K) \|_{\ell^{p'}} < \frac{1}{K}.
\]
Since $R_K (\alpha_K) = \alpha_K$, we have obtained a weakly null sequence $(\alpha_K)\subset B_{\ell^{p'}}$ such that
\[
\lim_{K\to\infty} \sup_{\|a\|_{\ell^p}=1} C_K(|g|;a) = \lim_{K\to\infty} \| H^*_{g,\gamma} R_K \|_{\ell^{p'}\to\ell^{p'}}  \leq  \lim_{K\to\infty} \| H^*_{g,\gamma} \alpha_K \|_{\ell^{p'}} \leq \textrm{wem} (H^*_{g,\gamma} ) \leq \| H^*_{g,\gamma} \|_{e,\ell^{p'}\to\ell^{p'}},
\] 
and we are done.
\end{proof}

\section{Acknowledgments}
The author was financially supported by the Magnus Ehrnrooth Foundation.

\bibliographystyle{abbrv}
\bibliography{bibliography}

\end{document}